\documentclass[11pt]{article}
\usepackage[left=1in,top=1in,right=1in,bottom=1in,head=.1in,nofoot]{geometry}

\usepackage{setspace,url,bm,amsmath} % For double-spacing, URL font, math symbols

\usepackage[numbers]{natbib}
\usepackage{graphicx} % Graphics scaling
\usepackage{bbm}
\usepackage{latexsym}
\usepackage{caption}
\usepackage[margin=20pt]{subcaption}
\usepackage{hyperref}
\usepackage[]{algorithm2e}
\usepackage[all,import]{xy}

\usepackage[table]{xcolor}

\newcommand{\GG}[1]{}

\usepackage{amsthm}
\usepackage{amssymb}
\usepackage{amsmath}
\usepackage{color}

\usepackage{comment}
\theoremstyle{definition}

\newtheorem*{theorem*}{Theorem}
\newtheorem{theorem}{Theorem}

\newtheorem{lemma}{Lemma}

\newtheorem*{corollary*}{Corollary}

\usepackage{natbib} % ASA citation style
\bibpunct{(}{)}{;}{a}{}{,} % ASA citation style

\usepackage{color}
\usepackage{listings}

\def\ind{\begin{picture}(9,8)
         \put(0,0){\line(1,0){9}}
         \put(3,0){\line(0,1){8}}
         \put(6,0){\line(0,1){8}}
         \end{picture}
        }

    \DeclareMathOperator{\PP}{pr}

\def\RD{\textsc{RD}}
\def\RR{\textsc{RR}}

\def\P{\textnormal{pr}}
\def\true{\textnormal{true}}   
\def\BF{\textsc{BF}} 

\def\NDE{\textsc{NDE}}
\def\NIE{\textsc{NIE}}
\def\TE{\textsc{TE}}
 
\def\obs{\textnormal{obs}}

    \def\MR{\textsc{MR}}
    \def\HR{\textsc{HR}}
    \def\pa{\textnormal{pa}}    
    \def\S{Section}

\begin{document}
\doublespacing

\title{\bf Sharp sensitivity bounds for mediation under unmeasured mediator-outcome confounding}

\author{Peng Ding and Tyler J. VanderWeele\\
Department of Epidemiology, Harvard T. H. Chan School of Public Health, \\
Boston, Massachusetts 02115, U.S.A. \\
\url{pengdingpku@gmail.com} and \url{tvanderw@hsph.harvard.edu}}
\date{}
\maketitle

\begin{abstract}
It is often of interest to decompose a total effect of an exposure into the component that acts on the outcome through some mediator and the component that acts independently through other pathways. Said another way, we are interested in the direct and indirect effects of the exposure on the outcome. Even if the exposure is randomly assigned, it is often infeasible to randomize the mediator, leaving the mediator-outcome confounding not fully controlled. We develop a sensitivity analysis technique that can bound the direct and indirect effects without  parametric assumptions about the unmeasured mediator-outcome confounding. 

\noindent {\bf Keywords:} Bounding factor; Causal inference; Collider; Natural direct effect; Natural indirect effect.
\end{abstract}

\section{Introduction}

Researchers often conduct mediation analysis to assess the extent to which an effect is mediated through some particular pathway and to which the effect of an exposure on the outcome operates directly. Mediation analysis initially developed within genetics and psychology based on linear structural equation models \citep{Wright::1934, Baron::1986}, and has been formalized by the notions of natural direct and indirect effects under the potential outcomes framework \citep{Robins::1992, Pearl::2001} and the decision-theoretic framework \citep{Didelez::2006, Geneletti::2007}. However, identification of natural direct and indirect effects used in that literature relies on strong assumptions, including that of no unmeasured mediator-outcome confounding \citep{Pearl::2001, Vanderweele::2010, Imai::2010}. Even if we can rule out unmeasured exposure-mediator and exposure-outcome confounding by randomly assigning the exposure, full control of mediator-outcome confounding is often impossible because it is infeasible to randomize the mediator. Therefore, it is crucial in applied mediation analyses to investigate the sensitivity of the conclusions with respect to unmeasured mediator-outcome confounding. Previous sensitivity analysis techniques either rely on restrictive modeling assumptions \citep{Imai::2010}, 
%or require sensitivity parameters involving counterfactual terms \citep{tchetgen2012semiparametric, vanderweele2014sensitivity}, 
or require specifying a large number of sensitivity parameters \citep{Vanderweele::2010}. Other literature \citep{Sjolander::2009, Robins::2010} provides bounds for natural direct and indirect effects without imposing assumptions, but these consider the most extreme scenarios and are often too wide to be useful in practice. We develop a sensitivity analysis technique which requires specifying only two sensitivity parameters, without any modeling assumptions or any assumptions about the type of the unmeasured mediator-outcome confounder or confounders.
Our results imply Cornfield-type inequalities \citep{Cornfield::1959, Ding::2014} that the unmeasured confounder must satisfy to reduce the observed natural direct effect to a certain level or explain it away.

\section{Notation and Framework for Mediation Analysis}
\label{sec::notation}

Let $A$ denote the exposure, $Y$ the outcome, $M$ the mediator, $C$ a set of observed baseline covariates not affected by the exposure, and $U$ a set of unmeasured baseline covariates not affected by the exposure. In order to define causal effects, we invoke the potential outcomes framework \citep{Neyman::1923, Rubin::1974} and apply this in the context of mediation \citep{Robins::1992, Pearl::2001}. If a hypothetical intervention on $A$ is well-defined, we let $Y_a$ and $M_a$ denote the potential values of the outcome and the mediator that would have been observed had the exposure $A$ been set to level $a$. If hypothetical interventions on $A$ and $M$ are both well-defined, we further let $Y_{am}$ denote the potential value of the outcome that would have been observed had the exposure $A$ been set to level $a$, and had the mediator $M$ been set to level $m$ \citep{Robins::1992, Pearl::2001}. Following \cite{Pearl::2009} and \cite{Vanderweele::2015}, we need the consistency assumption: $Y_a=Y, M_a=M$ if $A=a$; $Y_{am}=Y$ if $A=a$ and $M=m$, for all $a,m.$ 
We further need the composition assumption that $Y_{aM_a}=Y_a$ for $a=0,1$.

We will assume that the exposure $A$ is binary, but all of the results in this paper are also applicable to a categorical or continuous exposure and could be applied comparing any two levels of $A$. In the main text, we consider a binary outcome $Y$, but in \S \ref{sec::discussion} we note that all the results hold for count and continuous positive outcomes and time-to-event outcomes with rare events. The mediator $M$, the observed covariates $C$, and the unmeasured confounder or confounders $U$, can be of general types, i.e., categorical, continuous, or mixed; scalar or vectors. For notational simplicity, in the main text we assume that $(M,C,U)$ are categorical, and in the Supplementary Material we present results for general types.

On the risk ratio scale, the conditional natural direct and indirect effect, comparing the exposure levels $A=1$ and $A=0$ within the observed covariate level $C=c$, are defined as 
\begin{eqnarray}
\label{def::rr}
\NDE_{\RR|c}^\true = \frac{  \P(  Y_{1M_0} = 1 \mid c ) } {  \P(Y_{0M_0} = 1 \mid c) },\quad 
\NIE_{\RR|c}^\true = {  \P(  Y_{1M_1}=1\mid c )  \over \P( Y_{1M_0}=1 \mid c ) }.
\end{eqnarray}
The conditional natural direct effect compares the distributions of the potential outcomes, when the exposure level changes from $A=0$ to $A=1$ but the mediator is fixed at $M_0.$ The conditional natural indirect effect compares the distributions of the potential outcomes, when the exposure level is fixed at level $A=1$ but the mediator changes from $M_0$ to $M_1.$ The conditional total effect can be decomposed as a product of the conditional direct and indirect effects as follows:
\begin{eqnarray*}
\TE_{\RR|c}^\true 
= \frac{  \P(  Y_{1} = 1 \mid c ) } {  \P(Y_{0} = 1 \mid c) }
= \NDE_{\RR|c}^\true \times \NIE_{\RR|c}^\true.
\label{eq::decompose-rr}
\end{eqnarray*}

On the risk difference scale, the conditional natural direct and indirect effect are defined as 
\begin{eqnarray}
\NDE_{\RD|c}^\true &=& \P(  Y_{1M_0} = 1 \mid c ) - \P(Y_{0M_0} = 1\mid c),\\
\NIE_{\RD|c}^\true &=& \P(  Y_{1M_1}=1\mid c )  - \P( Y_{1M_0}=1 \mid c ),
\label{def::rd}
\end{eqnarray} 
and the conditional total effect has the decomposition
\begin{eqnarray*}
\TE_{\RD|c}^\true = \P(  Y_{1} = 1 \mid c )  -   \P(Y_{0} = 1 \mid c)= \NDE_{\RD|c}^\true + \NIE_{\RD|c}^\true . 
\label{eq::decompose-rd}
\end{eqnarray*}

\section{Identification of Conditional Natural Direct and Indirect Effects}
\label{sec::identification}

Here we follow  \citet{Pearl::2001}'s identification strategy for natural direct and indirect effects. A number of authors have provided other subtly different sufficient conditions \citep{Imai::2010, Vansteelandt::2012, Lendle::2013}.
Let $\ind$ denote independence of random variables.
To identify the conditional natural direct and indirect effects by the joint distribution of the observed variables $(A,M,Y,C)$, \citet{Pearl::2001} assumes that for all $a,a^*$ and $m$,
\begin{eqnarray}\label{eq::pearl-assume}
Y_{am} \ind A \mid  C,\quad
Y_{am} \ind M \mid  (A, C),\quad
M_a\ind A \mid C,\quad 
Y_{am}\ind M_{a^*}\mid C.
\end{eqnarray}
The four assumptions in (\ref{eq::pearl-assume}) require that the observed covariates $C$ control exposure-outcome confounding, control mediator-outcome confounding, control exposure-mediator confounding, and ensure cross-world counterfactual independence, respectively. In particular, on the risk ratio scale, we can identify the conditional natural direct and indirect effects by
\begin{eqnarray} 
\NDE^\obs_{\RR|c} &=&{ \sum_m   \P(Y=1\mid A=1,m,c) \P(m\mid A=0,c)  \over  \sum_m  \P(Y=1\mid A=0,m,c)  \P(m\mid A=0,c) } , 
\label{eq::identi-rr-nde}\\
\NIE_{\RR|c}^\obs &=&   
{
\sum_m \P(Y=1\mid A=1,m,c)\P(m\mid A=1,c)
\over \sum_m \P(Y=1\mid A=1,m,c)\P(m\mid A=0,c)
}.
\label{eq::identi-rr-nie}
\end{eqnarray}
On the risk difference scale, we can identify the conditional natural direct and indirect effects by
\begin{eqnarray}
\NDE^\obs_{\RD|c} &=& \sum_m \{  \P(Y=1\mid A=1,m,c)  - \P(Y=1\mid A=0,m,c) \} \P(m\mid A=0,c) ,
\label{eq::identi-rd-nde}\\
\NIE_{\RD|c}^\obs &=& \sum_m \P(Y=1\mid A=1,m,c) \{  \P(m\mid A=1,c) -\P(m\mid A=0,c) \}.
\label{eq::identi-rd-nie}
\end{eqnarray}
Proofs of (\ref{eq::identi-rr-nde})--(\ref{eq::identi-rd-nie}) may be found in \citet{Pearl::2001} and \citet{Vanderweele::2015}.

If we replace $Y_{aM_{a^*}}$ in the definitions \eqref{def::rr}--\eqref{def::rd} by $Y_{a,G_{a^*|c}}$, with $G_{a^*|c}$ a random draw from the conditional distribution $\P(M_{a^*}\mid c)$, then we can drop the cross-world counterfactual independence assumption $Y_{am}\ind M_{a^*}\mid C$ \citep{Vanderweele::2015}. This view is related to the decision-theoretic framework without using potential outcomes \citep{Didelez::2006, Geneletti::2007}. We show in the Supplementary Material that because the alternative frameworks lead to the same empirical identification formulas in (\ref{eq::identi-rr-nde})--(\ref{eq::identi-rd-nie}), all our results below can be applied.

\section{Sensitivity Analysis With Unmeasured Mediator-Outcome Confounding}
\label{sec::sensitivity}

\subsection{Unmeasured Mediator-Outcome Confounding}
Assumptions in (\ref{eq::pearl-assume}) are strong and untestable. If the exposure is randomly assigned given the values of the observed covariates $C$, as in completely randomized experiments or randomized block experiments, then the first and third assumptions (\ref{eq::pearl-assume}) hold automatically owing to the randomization. In observational studies, we may have background knowledge to collect adequate covariates $C$ to control the exposure-outcome and exposure-mediator confounding such that the first and third assumptions in (\ref{eq::pearl-assume}) are plausible. However, direct intervention on the mediator is often infeasible, and it may not be possible to randomize. Therefore, the second assumption in (\ref{eq::pearl-assume}), the absence of mediator-outcome confounding, may be violated in practice. Furthermore, the fourth assumption in (\ref{eq::pearl-assume}) cannot be guaranteed even under randomization of both $A$ and $M$, and thus it is fundamentally untestable \citep{Robins::2010}.

For sensitivity analysis, we assume that $(C, U)$ jointly ensure (\ref{eq::pearl-assume}), i.e.,
\begin{eqnarray}\label{eq::pearl-assume-U}
Y_{am} \ind A \mid  (C,U),\quad 
Y_{am}\ind M\mid (A, C, U),\quad 
M_a\ind A \mid (C,U),\quad 
Y_{am}\ind M_{a^*}\mid (C, U) .
\end{eqnarray}
When $C$ control the exposure-mediator and exposure-outcome confounding, we further assume
\begin{eqnarray}
\label{eq::randomization-treatment}
A\ind U\mid C.
\end{eqnarray}
%The assumptions in (\ref{eq::pearl-assume-U}) are weak because we do not impose any restrictions on the unmeasured confounders $U.$
The independence relationships in \eqref{eq::pearl-assume-U} impose no restrictions on the unmeasured confounders $U$, and they become assumptions if we require at least one of the sensitivity parameters introduced in \S \ref{sec::sensitivityparameters} be finite.
Figure \ref{fg::DAG} illustrates such a scenario with the assumptions in (\ref{eq::pearl-assume-U}) and (\ref{eq::randomization-treatment}) holding, where $U$ contains the common causes of the mediator and the outcome, and $A$ and $U$ are conditionally independent given $C.$ 
In \S \ref{sec::discussion} and the Supplementary Material, we comment on the applicability of our results under violations of the assumption in (\ref{eq::randomization-treatment}).

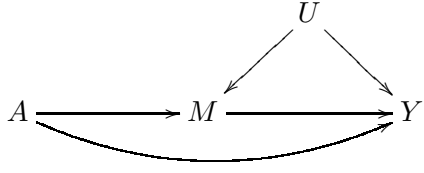
\begin{figure}[t]
\centering
$$
\begin{xy}
\xymatrix{
& & & U \ar[dl] \ar[dr] \\
A \ar[rr] \ar@/_1.5pc/[rrrr] & & M\ar[rr] & & Y}
\end{xy}
$$
\caption{Directed acyclic graph with mediator-outcome confounding within strata of observed covariates $C$.}\label{fg::DAG}
\end{figure}

Under the assumptions in (\ref{eq::pearl-assume-U}) and (\ref{eq::randomization-treatment}), we can express conditional natural direct and indirect effects using the joint distribution of $(A,M,Y,C,U)$. In particular, on the risk ratio scale, 
\begin{eqnarray} 
\NDE^\true_{\RR|c} &=&{ \sum_u \sum_m   \P(Y=1\mid A=1,m,c,u) \P(m\mid A=0,c,u) \P(u\mid c) 
\over  
\sum_u \sum_m  \P(Y=1\mid A=0,m,c,u)  \P(m\mid A=0,c,u)  \P(u\mid c)}, 
\label{eq::identi-rr-nde-u}\\
\NIE_{\RR|c}^\true &=&   
{
\sum_u \sum_m \P(Y=1\mid A=1,m,c,u)\P(m\mid A=1,c,u) \P(u\mid c)
\over 
\sum_u \sum_m \P(Y=1\mid A=1,m,c,u)\P(m\mid A=0,c,u) \P(u\mid c)
}.
\label{eq::identi-rr-nie-u}
\end{eqnarray}
On the risk difference scale,  
\begin{eqnarray}
\NDE^\true_{\RD|c} &= \sum_u \sum_m& \{  \P(Y=1\mid A=1,m,c,u)  - \P(Y=1\mid A=0,m,c,u) \} \nonumber \\
&&\times  \P(m\mid A=0,c,u) \P(u\mid c) , 
\label{eq::identi-rd-nde-u}\\
\NIE^\true_{\RD|c} &=\sum_u \sum_m& \P(Y=1\mid A=1,m,c,u)  \P(u\mid c) \nonumber \\
&&\times \{  \P(m\mid A=1,c,u) -\P(m\mid A=0,c,u) \}.
\label{eq::identi-rd-nie-u}
\end{eqnarray}
The proofs of (\ref{eq::identi-rr-nde-u})--(\ref{eq::identi-rd-nie-u}) follow from \citet{Pearl::2001} and \citet{Vanderweele::2015}. Unfortunately, however, the above four formulae about the conditional direct and indirect effects depend not only on the joint distribution of the observed variables $(A,M,Y,C)$ but also on the distribution of the unobserved variable $U$. In the following, we will give sharp bounds on the true conditional direct and indirect effects by the observed conditional natural direct and indirect effects and two measures of the mediator-outcome confounding that can be taken as sensitivity parameters.

\subsection{Sensitivity Parameters and the Bounding Factor}\label{sec::sensitivityparameters}

First, we introduce a conditional association measure between $U$ and $Y$ given $(A=1, M, C=c)$, and define our first sensitivity parameter as
$$
\RR_{UY|(A=1,M,c)} 
= \max_m \RR_{UY|(A=1,m,c)}
= \max_m
{ \max_{u} \P(Y=1\mid A=1,m,c,u)
\over 
\min_{u} \P(Y=1\mid A=1,m,c,u)
} ,
$$
where $ \RR_{UY|(A=1,m,c)}$ is the maximum divided by the minimum of the probabilities $\P(Y=1\mid A=1, m,c,u)$ over $u$. When $U$ is binary, $\RR_{UY|(A=1,m,c)}$ reduces to the usual conditional risk ratio of $U$ on $Y$, and $\RR_{UY|(A=1,M,c)} $ is the maximum of these conditional risk ratios over $m.$ If $U$ and $ Y$ are conditionally independent given $ (A,M,C)$, then $\RR_{UY|(A=1,M,c)}=1.$

Second, we introduce a conditional association measure between $A$ and $U$ given $M$. As illustrated in Figure \ref{fg::DAG}, although $A\ind U\mid C$, an association between $A$ and $U$ conditional on $M$ arises by conditioning on the common descendant $M$ of $A$ and $U$, also called the collider bias. Our second sensitivity parameter will assess the magnitude of this association generated by collider bias. We define our second sensitivity parameter as
\begin{eqnarray}
\label{eq::def2-1}
\RR_{AU|(M,c)}  = \max_m \RR_{AU|(m,c)}  
= \max_m \max_{u} {\P(u\mid A=1, m, c)  \over  \P( u\mid A=0,m,c  )}
,
\end{eqnarray}
where $\RR_{AU|(m,c)}$ is the maximum of the risk ratio of $A$ on $U$ taking value $u$ given $M=m$ and $C=c$. When $U$ is binary, $\RR_{AU|(m,c)}$ reduces to the usual conditional risk ratio of $A$ on $U$ given $M=m$ and $C=c$. The second sensitivity parameter can be viewed as the maximum of the collider bias ratios conditioning over stratum $M=m$. We give an alternative form 
%of $\RR_{AU|(m,c)}$ and provide a useful bound. We can write
%$$
%\RR_{AU|(m,c)}  
%=\max_u  \frac{\PP(A=1\mid m, c,u)}{\PP(A=0\mid  m, c,u)} \Big/  \frac{\PP(A=1\mid  m,c)}{\PP(A=0\mid m,c)} ,
%$$
%which is the maximum odds of $A$ conditional on $U$ divided by the odds of $A$ unconditional on $U.$ 
\begin{eqnarray}
\label{eq::def2-2}
\RR_{AU|(m,c)} 
=  \max_u  \frac{   \P(m\mid A=1, c,u)  }{\P(m\mid A=0, c, u)} \Big/ 
 \frac{   \P(m\mid A=1, c)  }{\P(m\mid A=0, c)} ,
\end{eqnarray}
which is the maximum conditional relative risk of $A$ on $M=m$ within stratum $U=u$ divided by the unconditional relative risk of $A$ on $M=m$. The relative risk unconditional on $U$ is identifiable from the observed data, and therefore the second sensitivity parameter depends crucially on the relative risk conditional on $U$.

Nonparametrically, we can specify the second sensitivity parameter using expression \eqref{eq::def2-1} or \eqref{eq::def2-2}. If we would like to impose parametric assumptions, e.g., $\P(m\mid a,c,u)$ follows a log-linear model, then it reduces to a function of the regression coefficients, which will depend explicitly on the $A$-$M$ and $U$-$M$ associations, as shown in the Supplementary Material.

To aid interpretation, Lemma A.3 in the Supplementary Material shows that
% a useful bound of $\RR_{AU|(m,c)} $:
$$
\RR_{AU|(m,c)} 
\leq
\max_{u\neq u'} \frac{ \PP(m\mid A=1,c,u) \PP(m\mid A=0,c, u')    }{ \PP(m\mid A=0,c,u) \PP(m\mid A=1, c,u')  },
$$
which measures the interaction of $A$ and $U$ on $M$ taking value $m$ given $C=c$ on the risk ratio scale \citep{Piegorsch::1994, Yang::1999}.
% If $M$ given $(A,C,U)$ follows a log-linear model without interactions, then $\RR_{AU|(m,c)} =1.$ 

To further aid specification of this second parameter we note that \citet{Greenland::2003} showed that, depending on the magnitude of the association, in most but not all settings, the magnitude of the ratio measure association relating $A$ and $U$ introduced by conditioning on $M$ is smaller than the ratios relating $A$ and $M$, and relating $U$ and $M$. Thus the lower of these two ratios can help specifying the second parameter. In particular, when the exposure is weakly associated with the mediator, the collider bias is small. If $A\ind M\mid C$, then the collider bias is zero, i.e., $\RR_{AU|(M,c)} =1.$

Finally, we introduce the bounding factor
\begin{eqnarray*}
\BF_{U|(M,c)} = 
\frac{ \RR_{AU|(M,c)}  \times \RR_{UY|(A=1,M,c)}    }
{  \RR_{AU|(M,c)}  +  \RR_{UY|(A=1,M,c)}   - 1},
\label{eq::bounding-factor}
\end{eqnarray*}
which is symmetric and monotone in both $\RR_{AU|(M,c)}$ and $\RR_{UY|(A=1,M,c)}$, and it
is no larger than either sensitivity parameter. If one of the sensitivity parameters equals unity, then the bounding factor also equals unity.
The bounding factor, a measure of the strength of unmeasured mediator-outcome confounding, plays a central role in bounding the natural direct and indirect effects in the following theorems.

\subsection{Bounding Natural Direct and Indirect Effects on the Risk Ratio Scale}

\begin{theorem}
\label{thm::NDE-RR}
Under the assumptions in (\ref{eq::pearl-assume-U}) and (\ref{eq::randomization-treatment}), the true conditional natural direct effect on the risk ratio scale has the sharp bound
$
 \NDE^\true_{\RR|c}  \geq   \NDE^\obs_{\RR|c}/  \BF_{U|(M,c)}.
$
\end{theorem}

The sharp bound is attainable when $U$ is binary, $\P(m\mid A=0,c)$ is degenerate, and some other conditions hold as discussed in the Supplementary Material.
%The bound is sharp in the sense that it is attainable when $U$ is a binary confounder. 
%See the Supplementary material for more details about the conditions for attaching the bound. 
Theorem \ref{thm::NDE-RR} provides an easy-to-use sensitivity analysis technique. After specifying the strength of the unmeasured mediator-outcome confounder, we can calculate the bounding factor, and then divide the point and interval estimates of the conditional natural direct effect by this bounding factor. This yields lower bounds on the conditional natural direct effect estimates. We can analogously apply the theorems below.

As shown in \S \ref{sec::notation}, the conditional total effect can be decomposed as the product of the conditional natural direct and indirect effects on the risk ratio scale, which, coupled with Theorem \ref{thm::NDE-RR}, implies the following bound on the conditional natural indirect effects.

\begin{theorem}
\label{thm::NIE-RR}
Under the assumptions in (\ref{eq::pearl-assume-U}) and (\ref{eq::randomization-treatment}), the true conditional natural indirect effect on the risk ratio scale has the sharp bound
$
\NIE^\true_{\RR|c}  
\leq  
\NIE_{\RR|c}^\obs
\times \BF_{U|(M,c)}.
$
\end{theorem}

Even if a researcher does not feel comfortable specifying the sensitivity parameters, one can still use Theorems \ref{thm::NDE-RR} and \ref{thm::NIE-RR} to report how large the sensitivity parameters would have to be for an estimate or lower confidence limit to lie below its null hypothesis value. We illustrate this in \S \ref{sec::cornfield} and \S \ref{sec::illustration} below.

If the natural direct effect averaged over $C$ is of interest, the true unconditional natural direct effect must be at least as large as the minimum of $\NDE^\obs_{\RR|c}/  \BF_{U|(M,c)}$ over $c$. If we further assume a common conditional natural direct effect among levels of $C$, as in the log-linear or logistic model for rare outcomes \citep[cf.][]{Vanderweele::2015}, then the true unconditional natural direct effect must be at least as large as the maximum of $\NDE^\obs_{\RR|c}/  \BF_{U|(M,c)}$ over $c$. A similar discussion holds for the unconditional natural indirect effect.

\subsection{Bounding Natural Direct and Indirect Effects on the Risk Difference Scale}

\begin{theorem} 
\label{thm::NDE-RD}
Under the assumptions in (\ref{eq::pearl-assume-U}) and (\ref{eq::randomization-treatment}), the true conditional natural direct effect on the risk difference scale has the sharp bound
$$
\NDE^\true_{\RD|c}  
\geq  \sum_m \PP(Y=1\mid A=1,m,c)   \PP(m\mid A=0,c)  / \BF_{U|(M,c)} 
-  \PP(Y  = 1 \mid A=0, c).
$$
\end{theorem}

Because the conditional total effect can be decomposed as the sum of the conditional natural direct and indirect effects on the risk difference scale as shown in \S \ref{sec::notation}, the identifiability of the conditional total effect and Theorem \ref{thm::NDE-RD} imply the following bound on the conditional natural indirect effect.

\begin{theorem} 
\label{thm::NIE-RD}
Under the assumptions in (\ref{eq::pearl-assume-U}) and (\ref{eq::randomization-treatment}), the true conditional natural indirect effect on the risk difference scale has the sharp bound
$$
\NIE^\true_{\RD|c}  
\leq    \PP(Y=1\mid A=1, c)-  \sum_m \PP(Y=1\mid A=1,m,c)   \PP(m\mid A=0,c)  / \BF_{U|(M,c)} . 
$$
\end{theorem}

Because of the linearity of the risk difference, the true unconditional direct and indirect effects can be obtained by averaging the bounds in Theorems \ref{thm::NDE-RD} and \ref{thm::NIE-RD} over the distribution of the observed covariates $C.$

\subsection{Cornfield-Type Inequalities for Unmeasured Mediator-Outcome Confounding}
\label{sec::cornfield}

We can equivalently state Theorem \ref{thm::NDE-RR} as the smallest value of the bounding factor to reduce an observed conditional natural direct effect to a true conditional causal natural direct effect, i.e.,
$
\BF_{U|(M,c)} \geq  \NDE^\obs_{\RR|c}/   \NDE^\true_{\RR|c}  , 
$
which further implies the following Cornfield-type inequalities \citep{Cornfield::1959, Ding::2014}.

\begin{theorem}\label{thm::cornfield}
Under the assumptions in (\ref{eq::pearl-assume-U}) and (\ref{eq::randomization-treatment}), to reduce an observed conditional natural direct effect $\NDE^\obs_{\RR|c}$ to a true conditional natural direct effect $\NDE^\true_{\RR|c} $, both $\RR_{AU|(M,c)}$ and $ \RR_{UY|(A=1,M,c)} $ must exceed $ \NDE^\obs_{\RR|c}/   \NDE^\true_{\RR|c}  ,$ and the largest of them must exceed
\begin{eqnarray} 
\label{eq::high}
\left[   \NDE^\obs_{\RR|c} + \left\{    \NDE^\obs_{\RR|c} ( \NDE^\obs_{\RR|c}- \NDE^\true_{\RR|c}  ) \right\}^{1/2}  \right] 
\Big/  \NDE^\true_{\RR|c}  . 
\end{eqnarray}
\end{theorem}

To explain away an observed conditional natural direct effect $\NDE^\obs_{\RR|c}$, i.e.,  $\NDE^\true_{\RR|c} =1$, 
both sensitivity parameters must exceed $\NDE^\obs_{\RR|c}$, and the maximum of them must exceed $\NDE^\obs_{\RR|c}+\{ \NDE^\obs_{\RR|c}  (\NDE^\obs_{\RR|c} - 1) \}^{1/2}.$ In the Supplementary Material, we present the inequalities derived from Theorem \ref{thm::NDE-RD} on the risk difference scale.

\section{Illustration}
\label{sec::illustration}

\citet{Vanderweele::2012genetic} conducted mediation analysis to assess the extent to which the effect of variants on chromosome 15q25.1 on lung cancer is mediated through smoking and to which it operates through other causal pathways. The exposure levels correspond to changes from 0 to 2 C alleles, smoking intensity is measured by the square root of cigarettes per day, and the outcome is the lung cancer indicator. The analysis of \citet{Vanderweele::2012genetic} was on the odds ratio scale using a lung cancer case-control study, but for a rare disease the odds ratios approximate risk ratios. After controlling for observed sociodemographic covariates, they found that the natural direct effect estimate is $1.72$ with $95\%$ confidence interval $[1.34, 2.21]$, and the natural indirect effect estimate  is $1.03$ with $95\%$ confidence interval $[0.99, 1.07]$. Their analysis used logistic regression models, requiring all the odds ratios be the same across different levels of the measured covariates.

The evidence for the indirect effect is weak because the confidence interval covers the null of no effect. However, the direct effect deviates significantly from the null. According to \S \ref{sec::cornfield}, to reduce the point estimate of the conditional natural direct effect to be below unity, both $\RR_{AU|(M,c)}$ and $\RR_{UY|(A=1,M,c)} $ must exceed $1.72$, and the maximum of them must exceed $1.72+(1.72\times 0.72)^{1/2} = 2.83$. Intuitively, for a binary confounder $U$ under parametric models with main effects, to explain away the direct effect estimate it would generally have to \citep[][cf. Supplementary Material]{Greenland::2003} increase the likelihood of $Y$ and increase $M$ by at least $1.72$-fold, and it would have to increase at least one of $Y$ or $M$ by $2.83$-fold. To reduce the lower confidence limit to be below unity, both sensitivity parameters must exceed $1.34$, and the maximum of them must exceed $1.34+(1.34\times 0.34)^{1/2} = 2.02.$
Intuitively, for a binary confounder $U$ under parametric models with main effects, to explain away the lower confidence limit for the direct effect it would generally have to \citep[][cf. Supplementary Material]{Greenland::2003} increase the likelihood of  $Y$ and increase $M$ by at least $1.34$-fold, and it would have to increase at least one of $Y$ or $M$ by $2.02$-fold. 
This would constitute fairly substantial confounding.

Previous studies found that the exposure-mediator association in this context is weak \citep{Saccone::2010}. Suppose the risk ratio relating $A$ and $M$ is less than $1.40$. If we assume that the collider bias is smaller than this in magnitude, e.g., $\RR_{AU|(M,c)} \leq 1.40$, as indicated by \citet{Greenland::2003}, then $\RR_{UY|(A=1,M,c)}$ must be at least as large as $11.47$ to reduce the point estimate to be below unity, and be at least as large as $8.93$ to reduce the lower confidence limit to be below unity. In general, when $\RR_{AU|(M,c)}$ is relatively small, we require an extremely large $\RR_{UY|(A=1,M,c)}$ to reduce the conditional natural direct effect estimate to be below unity. In fact, if $\RR_{AU|(M,c)}$ is smaller than the lower confidence limit of the conditional natural direct effect, it is impossible to reduce it to be below unity because the bounding factor is always smaller than $\RR_{AU|(M,c)}$.

\section{Discussion}
\label{sec::discussion}

Theorems \ref{thm::NDE-RR}--\ref{thm::cornfield} are most useful when the conditional natural direct effect is larger than unity. We can also simply relabel the exposure levels and all the results still hold.

In \S \ref{sec::sensitivity}, we derived sensitivity analysis formulae for causal parameters on the risk ratio and risk difference scales. If we have rare outcomes as in most case-control studies, we can approximate causal parameters on the odds ratio scale by those on the risk ratio scale, and all the results about risk ratio also apply to odds ratio. We have illustrated this in \S \ref{sec::illustration}.
Furthermore, we comment in the Supplementary Material that similar results also hold for count and continuous positive outcomes and rare time-to-event outcomes, if we replace the relative risks on the outcome by the hazard ratios and mean ratios.

The assumption $A\ind U\mid C$ may be violated if $U$ affects $(A,M,Y)$ simultaneously, i.e., unmeasured exposure-mediator, exposure-outcome, and mediator-outcome confounding all exist. Even if $A\ind U\mid C$ is violated, we show in the Supplementary Material that Theorems \ref{thm::NDE-RR} and \ref{thm::NDE-RD} can be interpreted as the bounds of the conditional natural direct effects for the unexposed population, which is also of interest in other contexts \citep{Vansteelandt::2012, Lendle::2013}.

\section*{Acknowledgement}
The authors thank the editor, associate editor and two referees for helpful comments. 
This research was funded by National Institutes of Health, U.S.A. 
T. J. VanderWeele is also affiliated with the Department of Biostatistics at the Harvard T. H. Chan School of Public Health.

%
%\section*{Supplementary Material}
%\label{SM}
%Supplementary material available at \Bka~online includes proofs of the theorems and more details about the discussion in \S \ref{sec::identification}, \S \ref{sec::sensitivityparameters}, \S \ref{sec::cornfield} and \S \ref{sec::discussion}.
%
%

\bibliographystyle{plainnat}
\bibliography{sensitivitymediation}

\begin{thebibliography}{27}
\providecommand{\natexlab}[1]{#1}
\providecommand{\url}[1]{\texttt{#1}}
\expandafter\ifx\csname urlstyle\endcsname\relax
  \providecommand{\doi}[1]{doi: #1}\else
  \providecommand{\doi}{doi: \begingroup \urlstyle{rm}\Url}\fi

\bibitem[Baron and Kenny(1986)]{Baron::1986}
Reuben~M Baron and David~A Kenny.
\newblock The moderator-mediator variable distinction in social psychological
  research: {C}onceptual, strategic, and statistical considerations.
\newblock \emph{J. Pers. Soc. Psychol.}, 51:\penalty0 1173--1182, 1986.

\bibitem[Cornfield et~al.(1959)Cornfield, Haenszel, Hammond,
  et~al.]{Cornfield::1959}
J.~Cornfield, W.~Haenszel, E.~C. Hammond, et~al.
\newblock Smoking and lung cancer: {R}ecent evidence and a discussion of some
  questions.
\newblock \emph{J. Natl. Cancer. Inst.}, 22:\penalty0 173--203, 1959.

\bibitem[Didelez et~al.(2006)Didelez, Dawid, and Geneletti]{Didelez::2006}
Vanessa Didelez, A.~Philip Dawid, and Sara Geneletti.
\newblock Direct and indirect effects of sequential treatments.
\newblock In R.~Dechter and T.~Richardson, editors, \emph{Proceedings of the
  22nd Conference on Uncertainty in Artificial Intelligence}, pages 138--146.
  pp. 138--146. Corvallis: Association for Uncertainty in Artificial
  Intelligence Press, 2006.

\bibitem[Ding and VanderWeele(2014)]{Ding::2014}
Peng Ding and Tyler~John VanderWeele.
\newblock Generalized {C}ornfield conditions for the risk difference.
\newblock \emph{Biometrika}, 101:\penalty0 971--977, 2014.

\bibitem[Geneletti(2007)]{Geneletti::2007}
Sara Geneletti.
\newblock Identifying direct and indirect effects in a non-counterfactual
  framework.
\newblock \emph{J. R. Statist. Soc. \rm{B}}, 69:\penalty0 199--215, 2007.

\bibitem[Greenland(2003)]{Greenland::2003}
Sander Greenland.
\newblock Quantifying biases in causal models: {C}lassical confounding vs
  collider-stratification bias.
\newblock \emph{Epidemiology}, 14:\penalty0 300--306, 2003.

\bibitem[Imai et~al.(2010)Imai, Keele, and Yamamoto]{Imai::2010}
Kosuke Imai, Luke Keele, and Teppei Yamamoto.
\newblock Identification, inference and sensitivity analysis for causal
  mediation effects.
\newblock \emph{Stat. Sci.}, 25:\penalty0 51--71, 2010.

\bibitem[Imbens(2003)]{imbens2003sensitivity}
Guido~W Imbens.
\newblock Sensitivity to exogeneity assumptions in program evaluation.
\newblock \emph{Am. Econ. Rev.}, pages 126--132, 2003.

\bibitem[Lendle et~al.(2013)Lendle, Subbaraman, and van~der Laan]{Lendle::2013}
Samuel~D Lendle, Meenakshi~S Subbaraman, and Mark~J van~der Laan.
\newblock Identification and efficient estimation of the natural direct effect
  among the untreated.
\newblock \emph{Biometrics}, 69:\penalty0 310--317, 2013.

\bibitem[Lin et~al.(1998)Lin, Psaty, and Kronrnal]{Lin::1998}
D.~Y. Lin, B.~M. Psaty, and R.~A. Kronrnal.
\newblock Assessing the sensitivity of regression results to unmeasured
  confounders in observational studies.
\newblock \emph{Biometrics}, 54:\penalty0 948--963, 1998.

\bibitem[Neyman(1923)]{Neyman::1923}
J.~Neyman.
\newblock On the application of probability theory to agricultural experiments.
  {E}ssay on principles. {S}ection 9.
\newblock \emph{Stat. Sci.}, 5:\penalty0 465--472, 1923.

\bibitem[Pearl(2001)]{Pearl::2001}
Judea Pearl.
\newblock Direct and indirect effects.
\newblock In J.~S. Breese and D.~Koller, editors, \emph{Proceedings of the 17th
  Conference on Uncertainty in Artificial Intelligence}, pages 411--420. pp.
  411--420. San Francisco: Morgan Kaufmann Publishers Inc., 2001.

\bibitem[Pearl(2009)]{Pearl::2009}
Judea Pearl.
\newblock \emph{Causality: Models, Reasoning and Inference}.
\newblock Cambridge: Cambridge University Press, 2009.

\bibitem[Piegorsch et~al.(1994)Piegorsch, Weinberg, Taylor,
  et~al.]{Piegorsch::1994}
Walter~W Piegorsch, Clarice~R Weinberg, Jack~A Taylor, et~al.
\newblock Non-hierarchical logistic models and case-only designs for assessing
  susceptibility in population-based case-control studies.
\newblock \emph{Stat. Med.}, 13:\penalty0 153--162, 1994.

\bibitem[Robins and Greenland(1992)]{Robins::1992}
James~M Robins and Sander Greenland.
\newblock Identifiability and exchangeability for direct and indirect effects.
\newblock \emph{Epidemiology}, 3:\penalty0 143--155, 1992.

\bibitem[Robins and Richardson(2010)]{Robins::2010}
James~M Robins and Thomas~S Richardson.
\newblock Alternative graphical causal models and the identification of direct
  effects.
\newblock In P.~Shrout, editor, \emph{Causality and Psychopathology: Finding
  the Determinants of Disorders and Their Cures}, pages 103--158. pp. 103--158.
  Oxford: Oxford University Press, 2010.

\bibitem[Rosenbaum and Rubin(1983)]{Rosenbaum::1983JRSSB}
P.~R. Rosenbaum and D.~B. Rubin.
\newblock Assessing sensitivity to an unobserved binary covariate in an
  observational study with binary outcome.
\newblock \emph{J. Roy. Stat. Soc. B.}, 45:\penalty0 212--218, 1983.

\bibitem[Rubin(1974)]{Rubin::1974}
D.~B. Rubin.
\newblock Estimating causal effects of treatments in randomized and
  nonrandomized studies.
\newblock \emph{J. Educ. Psychol.}, 66:\penalty0 688--701, 1974.

\bibitem[Saccone et~al.(2010)Saccone, Culverhouse, Schwantes-An,
  et~al.]{Saccone::2010}
Nancy~L Saccone, Robert~C Culverhouse, Tae-Hwi Schwantes-An, et~al.
\newblock Multiple independent loci at chromosome 15q25.1 affect smoking
  quantity: {A} meta-analysis and comparison with lung cancer and {COPD}.
\newblock \emph{PLoS Genetics}, 6\penalty0 (8):\penalty0 e1001053, 2010.

\bibitem[Sj{\"o}lander(2009)]{Sjolander::2009}
Arvid Sj{\"o}lander.
\newblock Bounds on natural direct effects in the presence of confounded
  intermediate variables.
\newblock \emph{Stat. Med.}, 28:\penalty0 558--571, 2009.

\bibitem[VanderWeele(2015)]{Vanderweele::2015}
Tyler~J VanderWeele.
\newblock \emph{Explanation in Causal Inference: Methods for Mediation and
  Interaction}.
\newblock Oxford: Oxford University Press, 2015.

\bibitem[VanderWeele et~al.(2012)VanderWeele, Asomaning, Tchetgen,
  et~al.]{Vanderweele::2012genetic}
Tyler~J VanderWeele, Kofi Asomaning, Eric J~Tchetgen Tchetgen, et~al.
\newblock Genetic variants on 15q25.1, smoking, and lung cancer: {A}n
  assessment of mediation and interaction.
\newblock \emph{Am. J. Epidemiol.}, 175:\penalty0 1013--1020, 2012.

\bibitem[VanderWeele(2010)]{Vanderweele::2010}
Tyler~John VanderWeele.
\newblock Bias formulas for sensitivity analysis for direct and indirect
  effects.
\newblock \emph{Epidemiology}, 21:\penalty0 540--551, 2010.

\bibitem[VanderWeele(2013)]{Vanderweele::2013}
Tyler~John VanderWeele.
\newblock Unmeasured confounding and hazard scales: {S}ensitivity analysis for
  total, direct, and indirect effects.
\newblock \emph{Eur. J. Epidemiol.}, 28:\penalty0 113--117, 2013.

\bibitem[Vansteelandt and VanderWeele(2012)]{Vansteelandt::2012}
Stijn Vansteelandt and Tyler~J VanderWeele.
\newblock Natural direct and indirect effects on the exposed: {E}ffect
  decomposition under weaker assumptions.
\newblock \emph{Biometrics}, 68:\penalty0 1019--1027, 2012.

\bibitem[Wright(1934)]{Wright::1934}
Sewall Wright.
\newblock The method of path coefficients.
\newblock \emph{Ann. Math. Stat.}, 5:\penalty0 161--215, 1934.

\bibitem[Yang et~al.(1999)Yang, Khoury, Sun, and Flanders]{Yang::1999}
Quanhe Yang, Muin~J Khoury, Fengzhu Sun, and W~Dana Flanders.
\newblock Case-only design to measure gene-gene interaction.
\newblock \emph{Epidemiology}, 10:\penalty0 167--170, 1999.

\end{thebibliography}

\begin{center}
{\bf \Large Supplementary Material}
\end{center}

Appendix A presents three lemmas, which play key roles in later proofs of the theorems and are of independent interest. 
Appendix B contains the proofs of the theorems in \S 4.
Appendix C discusses the interpretation of the second sensitivity parameter under some parametric assumptions.
Appendix D includes extensions and technical details of the discussion in the main text.

    \def\MR{\textsc{MR}}
    \def\HR{\textsc{HR}}
    \def\pa{\text{pa}}

    \setcounter{equation}{0}
        \setcounter{section}{0}
    \setcounter{lemma}{0}
    \setcounter{theorem}{0}
    \setcounter{figure}{0}
        \setcounter{table}{0}

\renewcommand{\theequation}{A.\arabic{equation}}
\renewcommand{\thelemma}{A.\arabic{lemma}}
\renewcommand{\thesection}{Appendix~\Alph{section}}
\renewcommand{\thetheorem}{A.\arabic{theorem}}
\renewcommand{\thefigure}{A.\arabic{figure}}
\renewcommand{\thetable}{A.\arabic{table}}

\section{Lemmas}

\begin{lemma}
\label{lemma::basic-factorial-derivative}
Define 
$
h(x) = (c_1 x  +1) / (c_2 x + 1). 
$
If $c_1  >  c_2$, then $h'(x) > 0$, and $h(x)$ is increasing; if $c_1 \leq c_2$, then $h'(x) \leq 0$, and $h(x)$ is non-increasing.
\end{lemma}

\begin{lemma}
\label{lemma::cornfield-increasing}
If $x, y>1$, then $g(x,y) =(xy)/(x+y-1)$ is increasing in both $x$ and $y$.
\end{lemma}

The proofs of Lemmas \ref{lemma::basic-factorial-derivative} and \ref{lemma::cornfield-increasing} are straightforward.

Let $X$ be any random element of scalar or vectors. Let $F_1(dx)$ and $F_0(dx)$ be two probability measures defined on the domain of $X$. For example, they may be two distribution functions of $X$ conditional on different events. Assume that there exists a Radon--Nikodym derivative between $F_1(dx)$ and $F_0(dx)$, i.e., $ F_1(dx) / F_0(dx)  = \gamma(x) $. Define $\gamma = \max_x \gamma(x)$ as the maximum value of the Radon--Nikodym derivative over the domain of $X$. We have $\gamma \geq  1$; otherwise $\int F_1(dx) < \int F_0(dx)=1$, and $F_1(dx)$ cannot a probability measure. Let $r(x)$ be a nonnegative function of $x$, and define $ \delta = \max_x r(x) / \min_x r(x) \geq 1$ as the ratio of the maximum divided by the minimum of $r(x)$ over $x$.

\begin{lemma}
\label{lemma:fundamental}
If $\int r(x) F_1(dx) < \infty$ for $a=0,1$, then we have the sharp bound
\begin{eqnarray}
\label{eq::lemma3}
\frac{   \int r(x) F_1(dx)     }{  \int r(x) F_0(dx)} \leq \frac{ \gamma  \delta  }{  \gamma + \delta - 1}.
\end{eqnarray}
\end{lemma}

\begin{proof}
[of Lemma \ref{lemma:fundamental}]
First, we define
\begin{eqnarray*}
w_1 = \frac{   \int   \{   r(x) - \min_x r(x) \}  F_1(dx)    }{  \max_x r(x) - \min_x r(x)},  \quad 
1-w_1 = \frac{   \int   \{  \max_x r(x) -  r(x)  \}  F_1(dx)    }{  \max_x r(x) - \min_x r(x)},  \\
w_0 = \frac{   \int   \{   r(x) - \min_x r(x) \}  F_0(dx)    }{  \max_x r(x) - \min_x r(x)}, \quad 
1-w_0 = \frac{   \int   \{  \max_x r(x) -  r(x)  \}  F_0(dx)    }{  \max_x r(x) - \min_x r(x)}.
\end{eqnarray*}
Then the left hand side of the inequality in (\ref{eq::lemma3}) can be re-written as
\begin{eqnarray*}
\frac{   \int r(x) F_1(dx)     }{  \int r(x) F_0(dx)}&=&
\frac{    w_1 \max_x r(x)  + (1-w_1) \min_x r(x)    }
{    w_0 \max_x r(x)  + (1-w_0) \min_x r(x)       } ,
\end{eqnarray*}
with both the numerator and denominator expressed as convex combinations of the maximum and minimum values of $r(x).$

Second, we define $\Gamma = w_1/w_0$, and the left hand side of (\ref{eq::lemma3}) can be re-written as
\begin{eqnarray}
\label{eq::lemma3-function}
\frac{   \int r(x) F_1(dx)     }{  \int r(x) F_0(dx)}
=
 \frac{     \{  \max_x r(x) -  \min_x r(x) \}  w_1 + \min_x r(x)   }
{       \{  \max_x r(x) -  \min_x r(x) \}  /\Gamma \times w_1 + \min_x r(x)     }.
\end{eqnarray}
It is straightforward to show that $\Gamma$ is bounded from above by $\gamma$, because
\begin{eqnarray*}
\Gamma  
= \frac{     \int   \{   r(x) - \min_x r(x) \}  F_1(dx)    }{   \int   \{   r(x) - \min_x r(x) \} F_0(dx)  } 
= \frac{     \int   \{   r(x) - \min_x r(x) \}  \gamma(x)  F_0(dx)    }{   \int   \{   r(x) - \min_x r(x)\} F_0(dx)  }  
\leq  \max_x  \gamma(x)  = \gamma.
\end{eqnarray*}

If $\Gamma > 1$, the right-hand side of (\ref{eq::lemma3-function}) is increasing in $w_1$ according to Lemma \ref{lemma::cornfield-increasing}, and therefore it attains the maximum at $w_1=1$, i.e.,
\begin{eqnarray*}
\frac{   \int r(x) F_1(dx)     }{  \int r(x) F_0(dx)} 
&\leq& 
\frac{     \{  \max_x r(x) -  \min_x r(x) \}   + \min_x r(x)   }
{       \{  \max_x r(x) -  \min_x r(x) \}  /\Gamma   + \min_x r(x)     } \\
&=& \frac{  \max_x r(x) \times \Gamma    }{  \max_x r(x)  -  \min_x r(x) +  \min_x r(x)\times \Gamma   } 
= \frac{\Gamma \delta }{\Gamma + \delta - 1}
\leq  \frac{\gamma \delta }{\gamma + \delta - 1} ,
\end{eqnarray*}
where the last inequality follows from Lemma \ref{lemma::basic-factorial-derivative}.
If $\Gamma < 1$, the right-hand side of (\ref{eq::lemma3-function}) is decreasing in $w_1$ according to Lemma \ref{lemma::cornfield-increasing}, and therefore it attains the maximum at $w_1=0$, i.e.,
$$
\frac{   \int r(x) F_1(dx)     }{  \int r(x) F_0(dx)} 
\leq 
1 \leq  \frac{\gamma \delta }{\gamma + \delta - 1} ,
$$
where the last inequality follows from Lemma \ref{lemma::basic-factorial-derivative}. Therefore, (\ref{eq::lemma3}) holds for any $\Gamma$.
%Therefore, regardless of the value of $\Gamma$, (\ref{eq::lemma3}) always holds. 

The bound in Lemma \ref{lemma:fundamental} is sharp, in the sense that it is attainable when $X$ is a Bernoulli random variable. Without loss of generality, we assume $p_1 = F_1\{X=1 \}$, $q_1 = F_0\{X=1 \}$, $p_1/q_1 = \gamma \geq 1$, and $r(1) /  r(0) = \delta \geq 1$. The ratio in Lemma \ref{lemma:fundamental} becomes
$$
\frac{   \int r(x) F_1(dx)     }{  \int r(x) F_0(dx)}
=
\frac{ r(1)p_1 + r(0)(1-p_1)   }{  r(1)q_1 + r(0)(1-q_1) }
=
\frac{1 + (\delta  - 1)p_1}{ 1 + (\delta - 1) p_1 / \gamma }
\leq 
\frac{1 + (\delta  - 1)}{ 1 + (\delta - 1)  / \gamma }
=\frac{\delta\gamma}{\delta + \gamma -1},
$$ 
where the inequality above follows from Lemma \ref{lemma::basic-factorial-derivative} with $p_1$ taking value $1$.
\end{proof}

\section{Proofs of the Theorems in \S 4}

In the proofs below, we discuss $(C, U, M)$ of general types, and introduce general notation. For instance, summations will be replaced by integrations, $\max$ by $\sup$, and $\min$ by $\inf$. When $(C,U, M)$ are categorical, all formulae below reduce to those in the main text. 
%We define the first sensitivity parameter as
%$$
%\RR_{UY|(A=1,M,c)} 
%= \sup_m \RR_{UY|(A=1,m,c)}
%= \sup_m
%{ \sup_{u} \P(Y=1\mid A=1,m,c,u)
%\over 
%\inf_{u} \P(Y=1\mid A=1,m,c,u)
%} .
%$$
%We define the second sensitivity parameter as $\RR_{AU|(M,c)} 
%= \sup_m  \RR_{AU|(m,c)} $, where
%%$$
%%\RR_{AU|(M,c)} 
%%%= \sup_m  \RR_{AU|(m,c)}  
%%= 
%% \sup_{m } \sup_u { F(du\mid A=1,m,c)  \over F(du\mid A=0,m,c) } 
%%=
%% \sup_{m } \sup_u { \P(m\mid A=1,c,u)  \over \P(m\mid A=0,c,u) } \Big/ 
%%  { \P(m\mid A=1,c)  \over \P(m\mid A=0,c) },
%%$$ 
%$$
%  \RR_{AU|(m,c)}  
%= 
%  \sup_u { F(du\mid A=1,m,c)  \over F(du\mid A=0,m,c) } 
%=
% \sup_u { \P(m\mid A=1,c,u)  \over \P(m\mid A=0,c,u) } \Big/ 
%  { \P(m\mid A=1,c)  \over \P(m\mid A=0,c) }
%$$ 
%is expressed in two equivalent forms to aid interpretation.
%where $ \RR_{AU|(m,c)}$ is the maximum value of the Random--Nikodym derivative of $F(du\mid A=1,m,c)$ over $F(du\mid A=0,m,c)$ given $M=m$ and $C=c.$ 
%Given $M$, the conditional association between $U$ and $Y$ has the same form as the categorical case in the main text. 
%We define 
%$$
%\RR_{UY|(A=1,M,c)} 
%= \sup_m \RR_{UY|(A=1,m,c)}
%= \sup_m
%{ \sup_{u} \P(Y=1\mid A=1,m,c,u)
%\over 
%\inf_{u} \P(Y=1\mid A=1,m,c,u)
%} ,
%$$
% the same as the categorical case in the main text. 
When $\P(m\mid a,c,u)$ and $\P(y\mid a,m,c,u)$ follow parametric models, e.g., log-linear models for binary $M$ and $Y$, the above two sensitivity parameters reduce to functions of the model parameters or the regression coefficients. This fact helps interpret the sensitivity parameters for both discrete and continuous $C$ and $U.$

\begin{proof}
[of Theorem 1] 
The observed conditional natural direct effect is
\begin{eqnarray}
\NDE^\obs_{\RR|c}  
&=& { \int   \P( Y=1 \mid A=1,m,c) F(dm\mid A=0,c)  \over  \int   \P(Y=1\mid A=0,m,c)  F(dm\mid A=0,c) }  
\label{eq::nde-obs1}\\
&=& { \int \int   \P( Y=1 \mid A=1,m,c,u)  F(du\mid A=1,m,c)  F(dm\mid A=0,c) 
 \over  \int \int   \P(Y=1\mid A=0,m,c, u) F(du\mid A=0 , m, c)  F(dm\mid A=0,c) }  
 \label{eq::-nde-obs2} \\
 &=& { \int \int   \P( Y=1\mid A=1,m,c,u)  F(du\mid A=1,m,c)  F(dm\mid A=0,c) 
 \over  \int \int   \P(Y=1\mid A=0,m,c, u) F(dm, du\mid A=0 , c) } ,
 \label{eq::nde-obs3} 
\end{eqnarray} 
where (\ref{eq::nde-obs1}) follows from the definition, (\ref{eq::-nde-obs2}) follows from the law of total probability, and (\ref{eq::nde-obs3}) follows from the definition of the joint distribution of $(M,U)$.

The true conditional natural direct effect is
\begin{eqnarray}
\NDE_{\RR|c} ^\true 
&=&  \frac{ \int \int   \P(  Y_{1M_0} = 1 \mid  M_0 = m,  c, u )  F_{(M_0, U)} ( dm, du\mid c  ) } 
{  \int \int  \P(Y_{0M_0} = 1 \mid M_0 =m,  c, u)  F_{(M_0, U)} ( dm, du\mid c  )   }  
\label{eq::nde-2}\\
&=& \frac{ \int \int   \P(  Y_{1m} = 1 \mid  M_0 = m,  c, u )  F_{(M_0, U)} ( dm, du\mid c  ) } 
{  \int \int  \P(Y_{0m} = 1 \mid M_0 =m,  c, u)  F_{(M_0, U)} ( dm, du\mid c  )   }  
\label{eq::nde-3}\\
&=& \frac{ \int \int   \P(  Y_{1m} = 1 \mid    c, u )  F_{M_0}(dm\mid c,u)   F ( du\mid c  ) } 
{  \int \int  \P(Y_{0m} = 1 \mid    c, u)   F_{M_0}(dm\mid c,u)   F ( du\mid c  )  }  
\label{eq::nde-4}\\
&=&  \frac{ \int \int   \P(  Y_{1m} = 1 \mid A=1, m , c, u )  F_{M_0}(dm\mid A=0, c,u)   F ( du\mid c  ) } 
{  \int \int  \P(Y_{0m} = 1 \mid A=0,m,   c, u)   F_{M_0}(dm\mid A=0, c,u)   F ( du\mid c  )   }  
\label{eq::nde-5}\\
&=& \frac{ \int \int   \P(  Y= 1 \mid A=1, m,  c, u )  F(dm\mid A=0, c,u)   F ( du\mid c  ) } 
{  \int \int  \P(Y = 1 \mid A=0,m,   c, u)   F(dm\mid A=0, c,u)   F ( du\mid c  )   }  
\label{eq::nde-6}\\
&=&  \frac{ \int \int   \P(  Y= 1 \mid A=1, m,  c, u )  F(dm\mid A=0, c,u)   F ( du\mid A=0,  c  ) } 
{  \int \int  \P(Y = 1 \mid A=0,m,   c, u)   F(dm\mid A=0, c,u)   F ( du\mid A=0, c  )   }    
\label{eq::nde-7}\\
&=& \frac{ \int \int   \P(  Y= 1 \mid A=1, m,  c, u )  F(dm, du \mid A=0, c)     } 
{  \int \int  \P(Y = 1 \mid A=0,m,   c, u)   F(dm, du\mid A=0, c )   }     ,
\label{eq::nde-8}
\end{eqnarray}
where (\ref{eq::nde-2}) follows from the definition and the law of total probability, (\ref{eq::nde-3}) follows from consistency, (\ref{eq::nde-4}) follows from $Y_{am}\ind M_0\mid (C,U)$ and the definition of the joint distribution of $(M_0, U)$, (\ref{eq::nde-5}) follows from $Y_{am}\ind A\mid (C,U), Y_{am}\ind M\mid (A,C,U)$ and $A\ind M_0\mid (C,U)$, (\ref{eq::nde-6}) follows from consistency, (\ref{eq::nde-7}) follows from $A\ind U\mid C$, and (\ref{eq::nde-8}) follows from the definition of the joint distribution of $(M,U)$ given $A=0$ and $C=c.$

Therefore, the ratio $ \NDE^\obs_{\RR|c}  / \NDE^\true_{\RR|c} $ is bounded by
\begin{eqnarray}  
\frac{\NDE^\obs_{\RR|c} } { \NDE^\true_{\RR|c}     }  
&=& { \int \int   \P( Y=1\mid A=1,m,c,u)  F(du\mid A=1,m,c)  F(dm\mid A=0,c)  \over  
\int \int   \P(  Y= 1 \mid A=1, m,  c, u )  F(dm, du\mid A=0, c)
}  \nonumber\\
&=& { \int   \left\{    \int   \P( Y=1\mid A=1,m,c,u)  F(du\mid A=1,m,c)  \right\}  F(dm\mid A=0,c)  
\over  
\int \left\{  \int   \P(  Y= 1 \mid A=1, m,  c, u )  F(du\mid A=0, m,c) \right\} F(dm\mid A=0,c)   
} \nonumber\\
&\leq& \sup_m  
{
  \int   \P( Y=1\mid A=1,m,c,u)  F(du\mid A=1,m,c)
  \over 
  \int   \P(  Y= 1 \mid A=1, m,  c, u )  F(du\mid A=0, m,c)
} \label{eq::ratio-bound-max-over-m} 
.
\end{eqnarray}
For given values of $m$ and $c$, $\P( Y=1\mid A=1,m,c,u)$ is a nonnegative function of $u$, and $F(du\mid A=1,m,c)$ and $F(du\mid A=0,m,c)$ are two measures on the domain of $U.$ 
We apply Lemma \ref{lemma:fundamental}, and obtain
\begin{eqnarray}
{
  \int   \P( Y=1\mid A=1,m,c,u)  F(du\mid A=1,m,c)
  \over 
  \int   \P(  Y= 1 \mid A=1, m,  c, u )  F(du\mid A=0, m,c)
}
\leq
{    \RR_{AU|(m,c)}\times \RR_{UY|(A=1,m,c)}  
\over
  \RR_{AU|(m,c)} +  \RR_{UY|(A=1,m,c)}   -1    }. \nonumber \\
  \label{eq::ineq-2}
\end{eqnarray}
The bounds in (\ref{eq::ratio-bound-max-over-m}) and (\ref{eq::ineq-2}), and Lemma \ref{lemma::cornfield-increasing} altogether imply
\begin{eqnarray}
\frac{\NDE^\obs_{\RR|c} } { \NDE^\true_{\RR|c}     } 
&\leq&    
\sup_m  
{    \RR_{AU|(m,c)}\times \RR_{UY|(A=1,m,c)}  
\over
  \RR_{AU|(m,c)} +  \RR_{UY|(A=1,m,c)}   -1    }  = \sup_m \BF_{U|(m,c)}  \nonumber \\
  & \leq&  
  {   \RR_{AU|(M,c)}\times  \RR_{UY|(A=1,M,c)}  
\over
   \RR_{AU|(M,c)} +   \RR_{UY|(A=1,M,c)}   -1    }
   = \BF_{U|(M,c)}. \label{eq::ineq-3}
\end{eqnarray} 

In the above proof, we have inequalities at three places (\ref{eq::ratio-bound-max-over-m})--(\ref{eq::ineq-3}). First, (\ref{eq::ineq-3}) is attainable if $ \RR_{AU|(m,c)}$ and $ \RR_{UY|(A=1,m,c)} $ attain their maximum values at the same level of $m$. Second, (\ref{eq::ineq-2}) is attainable according to the sharpness of Lemma \ref{lemma:fundamental}. Third, (\ref{eq::ratio-bound-max-over-m}) is attainable, if $F(dm\mid A=0,c)$ has all mass on the value $m$ that attains the maximum  of $ \BF_{U|(m,c)} .$ Furthermore, the conditions for attaining these inequalities are compatible, implying that the bound in Theorem 1 is sharp. 
\end{proof}

\begin{proof}
[of Theorem 2]
Similar to (\ref{eq::nde-8}) in the proof of Theorem 1,  
\begin{eqnarray}
\NIE_{\RR|c} ^\true 
= \frac{ \int \int   \P(  Y= 1 \mid A=1, m,  c, u )  F(dm, du \mid A=1, c)     } 
{  \int \int  \P(Y = 1 \mid A=1,m,   c, u)   F(dm, du\mid A=0, c )   }     .
\label{eq::nie}
\end{eqnarray}
Because the conditional total effect can be decomposed as the product of the conditional natural direct and indirect effects on the risk ratio scale, formulas (\ref{eq::nde-8}) and (\ref{eq::nie}) imply
\begin{eqnarray}
\TE_{\RR|c}^\true &=& \NDE_{\RR|c} ^\true\times \NIE_{\RR|c} ^\true 
= \frac{ \int \int   \P(  Y= 1 \mid A=1, m,  c, u )  F(dm, du \mid A=1, c)     } 
{  \int \int  \P(Y = 1 \mid A=0,m,   c, u)   F(dm, du\mid A=0, c )   } \nonumber \\
&=& { \P(Y=1\mid A=1,c) \over  \P(Y=1\mid A=0,c)} 
= \TE_{\RR|c}^\obs =  \NDE_{\RR|c} ^\obs \times  \NIE_{\RR|c} ^\obs.
\label{eq::total-effect-product}
\end{eqnarray} 
Therefore, once we obtain the bound on the conditional natural direct effect, we can immediately obtain the bound on the conditional natural indirect effect. According to (\ref{eq::total-effect-product}) and Theorem 1,  
\begin{eqnarray*}
\NIE^\true_{\RR|c}   
=  \NIE_{\RR|c}^\obs \times \frac{ \NDE_{\RR|c}^\obs  }{ \NDE_{\RR|c}^\true  }
\geq   \NIE^\obs_{\RR|c}    \times \BF_{U|(M,c)}.
\end{eqnarray*}

The bound is sharp according to the proof of Theorem 1.
\end{proof}

%
%
%For the following proof, we need to introduce the conditional natural direct effect on the risk difference scale for the unexposed population:
%\begin{eqnarray}
%\NDE^\true_{\RD|(A=0,c)}
%=  \P(  Y_{1M_0} = 1 \mid A=0,  c ) -  \P(Y_{0M_0} = 1 \mid A=0, c)  .
%\label{eq::cnde-rd-unexposed}
%\end{eqnarray}
%

\begin{proof}
[of Theorem 3]
%We first prove the general conclusion without the assumption $A\ind U\mid C.$
%The true conditional natural direct effect on the risk difference scale satisfies:
We can write the conditional natural direct effect, $\NDE^\true_{\RD|c} $, as
\begin{eqnarray}  
%&&\NDE^\true_{\RD|c}  \nonumber  \\
&&  \P(  Y_{1M_0} = 1 \mid c ) -  \P(Y_{0M_0} = 1 \mid c)  \label{eq::rd-nde-p1}\\
&=&     \int  \P(Y=1\mid A=1,m,c)   F(dm\mid A=0,c)  
\Big/{  \int  \P(Y=1\mid A=1,m,c)   F(dm\mid A=0,c) \over  \P(  Y_{1M_0} = 1 \mid  c )  } \nonumber  \\
&&-  \P(Y  = 1 \mid A=0, c)  \label{eq::rd-nde-p2} \\
&=&  \int  \P(Y=1\mid A=1,m,c)   F(dm\mid A=0,c)  
\Big/{  \NDE^\obs_{\RR|c}    \over \NDE^\true_{\RR|c}  }
-  \P(Y  = 1 \mid A=0, c)\label{eq::rd-nde-p3}   \\
&\geq & \int  \P(Y=1\mid A=1,m,c)   F(dm\mid A=0,c)  / \BF_{U|(M,c)} 
-  \P(Y  = 1 \mid A=0, c)  . \label{eq::rd-nde-p4} 
%&=&  \int \{   \P(Y=1\mid A=1,m,c) / \BF_{U|(M,c)}  -    \P(Y=1\mid A=0,m,c)  \}  F(dm\mid A=0,c)  , \nonumber \\
%&&\label{eq::rd-nde-p5} 
\end{eqnarray}  
where (\ref{eq::rd-nde-p1}) is by the definition, (\ref{eq::rd-nde-p2}) follows from the proof of Theorem 1, (\ref{eq::rd-nde-p3}) follows from the proof of Theorem 1, and (\ref{eq::rd-nde-p4}) follows from Theorem 1.

The bound is sharp according to the proof of Theorem 1.
\end{proof}

%
%For the following proof, we need to introduce the conditional natural indirect effect on the risk difference scale for the unexposed population
%$$
%\NIE_{\RD|(A=0,c)}^\true =  \P(  Y_{1M_1}=1\mid A=0, c )  -  \P( Y_{1M_0}=1 \mid A=0, c )  ;
%$$
%and the conditional total effect on the risk difference scale for the unexposed population
%$$
%\TE_{\RD|(A=0,c)}^\true = \P(  Y_1=1\mid A=0, c )  -  \P( Y_0 =1 \mid A=0, c ) .
%$$
%
%

\begin{proof}
[of Theorem 4]
Similar to (\ref{eq::nde-8}), (\ref{eq::nie}) and (\ref{eq::total-effect-product}) on the risk ratio scale,  
\begin{eqnarray}
\TE_{\RR|c}^\true =   \NIE^\true_{\RD|c}  + \NIE^\true_{\RD|c}  =   \P(Y=1\mid A=1,c) - \P(Y=1\mid A=0,c)  .
\label{eq::total-effect-sum}
\end{eqnarray}
Therefore, once we obtain the bound on the conditional natural direct effect, we can immediately obtain the bound on the conditional natural indirect effect. According to Theorem 2, 
%we can write $\NIE^\true_{\RD|c}$ as
\begin{eqnarray*}
&&\NIE^\true_{\RD|c} =  \TE_{\RD|c}^\true  -    \NDE^\true_{\RD|c}  \\
&\leq & \{  \P(Y=1\mid A=1,c) - \P(Y=1\mid A=0,c) \}  \\
&&- \left\{  \int  \P(Y=1\mid A=1,m,c)   F(dm\mid A=0,c)  / \BF_{U|(M,c)} 
-  \P(Y  = 1 \mid A=0, c)     \right\} \\
&=& \P(Y=1\mid A=1,c) -  \int  \P(Y=1\mid A=1,m,c)   F(dm\mid A=0,c)  / \BF_{U|(M,c)}.
\end{eqnarray*}

The bound is sharp according to the proof of Theorem 1.
\end{proof}

\begin{proof}
[of Theorem 5]
Theorem 1 is equivalent to 
\begin{equation}\label{eq::proof-cornfield}
\BF_{U|(M,c)}  = g\left\{  \RR_{AU|(M,c)},  \RR_{UY|(A=1,M,c)}  \right\}    
\geq   \NDE^\obs_{\RR|c}/   \NDE^\true_{\RR|c} ,
\end{equation}
where $g(x,y) = (xy)/(x+y-1)$ was defined in Lemma \ref{lemma::cornfield-increasing}.
According to Lemma \ref{lemma::cornfield-increasing}, $\BF_{U|(M,c)}$ is increasing in both $\RR_{AU|(M,c)}$ and $  \RR_{UY|(A=1,M,c)}$. Letting $\RR_{AU|(M,c)} $ in (\ref{eq::proof-cornfield}) go to infinity, we have 
$
 \RR_{UY|(A=1,M,c)}
\geq   \NDE^\obs_{\RR|c}/   \NDE^\true_{\RR|c} .
$ 
By symmetry, we also have
$
\RR_{AU|(M,c)}
\geq   \NDE^\obs_{\RR|c}/   \NDE^\true_{\RR|c} .
$
Let $\RR_{\max} = \max (  \RR_{AU|(M,c)},   \RR_{UY|(A=1,M,c)} ) $. According to Lemma \ref{lemma::cornfield-increasing}, we have
$$
\frac{  \RR_{\max}^2    }{ 2 \RR_{\max} - 1} \geq \frac{ \NDE^\obs_{\RR|c} } {   \NDE^\true_{\RR|c}  } .
$$
Solving the inequality of $\RR_{\max}$ above, we obtain the high threshold in Theorem 5.
%\begin{eqnarray*}
%\RR_{\max} \geq  
%\left[   \NDE^\obs_{\RR|c} + \left\{    \NDE^\obs_{\RR|c} ( \NDE^\obs_{\RR|c}- \NDE^\true_{\RR|c}  ) \right\}^{1/2}  \right] 
%\Big/  \NDE^\true_{\RR|c}  .
%\end{eqnarray*}
\end{proof}

\section{The Second Sensitivity Parameter $\RR_{AU|(M,c)} $}

\subsection{A Lemma About Collider Bias}

\begin{lemma}
\label{lemma-collider-bias}
In a directed acyclic graph with vertices $(A,U,M)$, if $A\ind U$ and $M$ is a collider, then the risk ratio of $A$ on $U$ given $M=m$ satisfies
\begin{eqnarray*}
\RR_{AU|m} 
&=& 
\max_u  \frac{\PP(u\mid A=1, m)}{\PP(u\mid A=0,m)}
=
\max_u  \frac{\PP(A=1\mid u, m)}{\PP(A=0\mid u, m)} \Big/  \frac{\PP(A=1\mid  m)}{\PP(A=0\mid m)} \\
&\leq&
\max_{u\neq u'} \frac{ \PP(m\mid A=1,u) \PP(m\mid A=0, u')    }{ \PP(m\mid A=0,u) \PP(m\mid A=1, u')  } .
\end{eqnarray*} 
\end{lemma}

\begin{proof}
[of Lemma \ref{lemma-collider-bias}]
Bayes' Theorem gives the first line of Lemma \ref{lemma-collider-bias} and 
\begin{eqnarray}\label{eq::bayes}
\RR_{AU|m} 
=
\max_u  \frac{\P(m\mid A=1, u) \P(u\mid A=1) /  \P(m\mid A=1) }{ \P(m\mid A=0, u) \P(u\mid A=0) /  \P(m\mid A=0)  } .
\end{eqnarray}
Applying $A\ind U$ and the law of total probability to \eqref{eq::bayes}, we have
\begin{eqnarray}\label{eq::total}
\RR_{AU|m} 
=
\max_u   \frac{\P(m\mid A=1, u)   /  \int  \P(m\mid A=1, u') F(du') }
{ \P(m\mid A=0, u) / \int  \P(m\mid A=0, u') F(du')  } .
\end{eqnarray}
We rearrange the terms in \eqref{eq::total}, and finally obtain
\begin{eqnarray*}
\RR_{AU|m} 
&=&
\max_u \left\{   \int   \frac{ \P(m\mid A=0, u') }{ \P(m\mid A=0, u)  }   F(du') \Big/
\int   \frac{ \P(m\mid A=1, u') }{ \P(m\mid A=1, u)  }    F(du') \right\} \\
&\leq& 
\max_u \max_{u'}   \left\{    \frac{ \P(m\mid A=0, u') }{ \P(m\mid A=0, u)  }   \Big/
   \frac{ \P(m\mid A=1, u') }{ \P(m\mid A=1, u)  }    \right\} \\
&=&
\max_{u\neq u'}      
 \frac{ \PP(m\mid A=1,u) \PP(m\mid A=0, u')    }{ \PP(m\mid A=0,u) \PP(m\mid A=1, u')  } .
\end{eqnarray*}
\end{proof}

\subsection{Under Some Parametric Assumptions}

The second sensitivity parameter can be rewritten as
$$
\RR_{AU|(M,c)} 
=  
\sup_m \sup_u  \frac{F(du\mid A=1, m,c)}{F(du\mid A=0, m,c)}
=
\max_m  \left\{  
\max_u  
\frac{F(dm\mid A=1, c, u)}{  F(dm\mid A=0, c, u)}
\Big /
\frac{F(dm\mid A=1, c )}{  F(dm\mid A=0, c)}
\right\}.
$$

Assume that $M$ is binary, and follows a log-linear model conditional on $(A,C,U)$:
$$
\P(M=1\mid a,c,u) = \exp\left\{   \beta_0 + \beta_1 a + \beta_2 c + \beta_3 u   \right\}.
$$

For simplicity, we assume that $U$ is a Bernoulli random variable with mean $1/2$ independent of $C$, as in \citet{Rosenbaum::1983JRSSB}, \citet{Lin::1998} and \citet{imbens2003sensitivity}. The cumulant generating function of $U$ is $K(t) = \log E(e^{t U}) = \log\{ (1+e^t) / 2\}$. Therefore, the marginal model of $M$ given $(A,C)$ also follows a log-linear model:
$$
\P(M=1\mid a,c) =  \exp\left\{   \beta_0 + K(\beta_3) + \beta_1 a + \beta_2 c \right\}.
$$
%From the data, we can identify $\beta_0' = \beta_0 + K(\beta_3) $, $\beta_1$, and $\beta_2$. 

For $m=1$, we have $\RR_{AU|(M=1,c)} = 1$ because in the log-linear model the conditional and unconditional relative risks of $A$ on $M$ are both $\exp^{\beta_1}$. For $m=0$, we have
$$
\RR_{AU|(M=0,c)} = 
\max_u  
\frac{    1-e^{\beta_0  + \beta_1 + \beta_2 c + \beta_3 u}   }{    1-e^{\beta_0    + \beta_2 c + \beta_3 u}     }
\Big / 
\frac{    1-e^{\beta_0' + \beta_1 + \beta_2 c}   }{ 1- e^{\beta_0'   + \beta_2 c}} .
$$
We need to find the maximum value on the right-hand side.
Ignoring some positive constants,
$$
\frac{\partial   }{ \partial \beta_3 }
\left(
\frac{    1-e^{\beta_0  + \beta_1 + \beta_2 c + \beta_3 u}   }{    1-e^{\beta_0    + \beta_2 c + \beta_3 u}     }
\right)
\propto
\beta_3 (1- e^{\beta_1}).
$$
%\begin{eqnarray*}
%&&\frac{\partial   }{ \partial \beta_3 }
%\left(
%\frac{    1-e^{\beta_0  + \beta_1 + \beta_2 c + \beta_3 u}   }{    1-e^{\beta_0    + \beta_2 c + \beta_3 u}     }
%\right) \\
%&=&
%\frac{      -e^{\beta_0  + \beta_1 + \beta_2 c + \beta_3 u}   \beta_3     ( 1-e^{\beta_0    + \beta_2 c + \beta_3 u}   )
%+    ( 1-e^{\beta_0  + \beta_1 + \beta_2 c + \beta_3 u} )  e^{\beta_0    + \beta_2 c + \beta_3 u}   \beta_3      }
%{   (   1-e^{\beta_0    + \beta_2 c + \beta_3 u}  )^2   } \\
%&\propto&
%\beta_3 (1- e^{\beta_1}).
%\end{eqnarray*}
If $\beta_1 \beta_3 \geq 0$, then the above derivative is non-positive, and 
$$
\RR_{AU|(M=0,c)} = 
\frac{    1-e^{\beta_0  + \beta_1 + \beta_2 c  }   }{    1-e^{\beta_0    + \beta_2 c  }     }
\Big / 
\frac{    1-e^{\beta_0' + \beta_1 + \beta_2 c}   }{ 1- e^{\beta_0'    + \beta_2 c}} 
=
\frac{    1-e^{\beta_0' - K(\beta_3)  + \beta_1 + \beta_2 c  }   }{    1-e^{\beta_0' - K(\beta_3)    + \beta_2 c  }     }
\Big / 
\frac{    1-e^{\beta_0' + \beta_1 + \beta_2 c}   }{ 1- e^{\beta_0'   + \beta_2 c}} 
.
$$
If $\beta_1\beta_3 <0$, then the above derivative is positive, and 
$$
\RR_{AU|(M=0,c)} = 
\frac{    1-e^{\beta_0  + \beta_1 + \beta_2 c + \beta_3 }   }{    1-e^{\beta_0    + \beta_2 c + \beta_3 }     }
\Big / 
\frac{    1-e^{\beta_0' + \beta_1 + \beta_2 c}   }{ 1- e^{\beta_0'   + \beta_2 c}} 
=
\frac{    1-e^{\beta_0' +\beta_3 -  K(\beta_3)   + \beta_1 + \beta_2 c   }   }{    1-e^{\beta_0' +\beta_3 - K(\beta_3)     + \beta_2 c  }     }
\Big / 
\frac{    1-e^{\beta_0' + \beta_1 + \beta_2 c}   }{ 1- e^{\beta_0'    + \beta_2 c}} .
$$

Because $\beta_0', \beta_1$ and $\beta_2$ can be identified from the marginal model of $\P(M=1\mid a,c)$ by the observed data, the second sensitivity parameter
$
\RR_{AU|(M,c)} = \max\{ 1 , \RR_{AU|(M=0,c)}  \}
$
reduces to a function of $\beta_3$, the confounder-mediation association. In practice, we can estimate $\beta_0', \beta_1$ and $\beta_2$ from the observed data, choose a range of plausible values of $\beta_3$, and compute the corresponding values of the second sensitivity parameter.

Without loss of generality, below we consider the case without covariates $C$, because in practice our analysis is often conducted within strata of $C$. Tables \ref{tb1} and \ref{tb2} show how the values of $\RR_{AU|(M,c)} / \exp(\beta_3)$ and $\RR_{AU|(M,c)} / \exp(\beta_1)$ vary with $\beta_3$, for different combinations of $(\beta_0, \beta_1)$. In all cases, the ratios are smaller than unity, which verify \citet{Greenland::2003}'s statement that the second sensitivity parameter is often smaller than both the exposure-mediator and confounder-mediator associations. We have tried many other values of the regression coefficients, and find that this is true in general. Therefore, we can often use the exposure-mediator and confounder-mediator associations as the upper bound of the second sensitivity parameter.

\begin{table}[t]
\centering
\caption{Values of $\RR_{AU|(M,c)} / \exp(\beta_3)$ for different combinations of $(\beta_0, \beta_1)$ with columns corresponding to different values of $\beta_3$}\label{tb1}
\begin{tabular}{rrrrrrrr}
  \hline
$(\beta_0, \beta_1)$ & 0.1 & 0.2 & 0.3 & 0.4 & 0.5 & 0.6 & 0.7 \\ 
  \hline
$(-2.3.0.2)$ & 0.91 & 0.82 & 0.74 & 0.68 & 0.61 & 0.56 & 0.50 \\ 
$(-2, 0.2)$ & 0.91 & 0.82 & 0.75 & 0.68 & 0.62 & 0.56 & 0.51 \\ 
$(-2.3, 0.4)$ & 0.91 & 0.82 & 0.75 & 0.68 & 0.62 & 0.56 & 0.51 \\ 
$(-2, 0.4)$ & 0.91 & 0.83 & 0.75 & 0.69 & 0.63 & 0.57 & 0.52 \\ 
$(-2.3, 0.7)$ & 0.91 & 0.83 & 0.76 & 0.70 & 0.64 & 0.58 & 0.54 \\ 
$(-2, 0.7)$& 0.92 & 0.84 & 0.77 & 0.71 & 0.66 & 0.61 & 0.56 \\ 
   \hline
\end{tabular}
\end{table}

\begin{table}[t]
\centering
\caption{Values of $\RR_{AU|(M,c)} / \exp(\beta_1)$ for different combinations of $(\beta_0, \beta_1)$ with columns corresponding to different values of $\beta_3$}\label{tb2}
\begin{tabular}{rrrrrrrr}
  \hline
$(\beta_0, \beta_1)$ & 0.1 & 0.2 & 0.3 & 0.4 & 0.5 & 0.6 & 0.7 \\ 
  \hline
$(-2.3.0.2)$& 0.82 & 0.82 & 0.82 & 0.82 & 0.83 & 0.83 & 0.83 \\ 
$(-2, 0.2)$ & 0.82 & 0.82 & 0.82 & 0.83 & 0.83 & 0.83 & 0.84 \\ 
$(-2.3, 0.4)$& 0.67 & 0.68 & 0.68 & 0.68 & 0.69 & 0.69 & 0.69 \\ 
$(-2, 0.4)$  & 0.67 & 0.68 & 0.68 & 0.69 & 0.69 & 0.70 & 0.71 \\ 
$(-2.3, 0.7)$ & 0.50 & 0.50 & 0.51 & 0.52 & 0.52 & 0.53 & 0.54 \\ 
$(-2, 0.7)$ & 0.50 & 0.51 & 0.52 & 0.53 & 0.54 & 0.55 & 0.56 \\ 
   \hline
\end{tabular}
\end{table}

\section{Extensions}

\subsection{Extension of Theorem 5 Based on Theorem 3}

\begin{theorem} 
\label{thm::cornfield-RD}
Under the assumptions in (6) and (7), the bounding factor must exceed
\begin{eqnarray}
\label{eq::cornfield-rd}
 \BF_{U|(M,c)}  \geq
 \frac{  \int  \PP(Y=1\mid A=1,m,c)   F(dm\mid A=0,c)      } {    \NDE^\true_{\RD|c} +    \PP(Y  = 1 \mid A=0, c) }.
\end{eqnarray}
Denote the right-hand side of (\ref{eq::cornfield-rd}) by $\Delta$. Then both $\RR_{AU|(M,c)}$ and $ \RR_{UY|(A=1,M,c)} $ must exceed $\Delta$, and the largest of them must exceed
$
\Delta + \left\{\Delta (\Delta - 1) \right\}^{1/2} . 
$
\end{theorem}

\begin{proof}
[of Theorem \ref{thm::cornfield-RD}]
The results follow from Theorem 3 and the proof of Theorem 5.
\end{proof}

When there is no conditional natural direct effect, i.e., $ \NDE^\true_{\RD|c} =0$ or equivalently $\NDE^\true_{\RR|c} =1$, Theorems 5 and \ref{thm::cornfield-RD} give the same Cornfield-type inequality.

\subsection{Time-to-Event Outcome With Rare Events}
Let $\lambda(t\mid x)$ be the hazard rate of the time-to-event outcome $Y$ at time $t$ conditional on some random variable $X.$
For time-to-event outcome, we need to modify the definition of confounder-outcome association given the exposure level $A=1$ and the mediator as
\begin{eqnarray}
\label{eq::hazard-ratio}
\HR_{UY|(A=1,M, c)} (t)= \sup_m 
{   \sup_{ u} \lambda(  t \mid A=1, m,  c, u )   \over \inf_{ u}  \lambda(  t \mid A=1, m,  c, u )  }
\end{eqnarray}
and correspondingly the bounding factor as
$$
\BF_{U|(M,c)} (t)= 
\frac{ \RR_{AU|(M,c)}  \times \HR_{UY|(A=1,M,c)} (t)   }
{  \RR_{AU|(M,c)}  +  \HR_{UY|(A=1,M,c)} (t)  - 1}.
$$

If we have rare events at the end of the study, the hazard ratio approximately satisfies 
$
\lambda(t) \approx \int \lambda(t\mid x) F(dx),
$
where $F(dx)$ is the distribution of random variable $X$ \citep{Vanderweele::2013}. This approximate linearity property is sufficient for us to manipulate hazard rates as probabilities in all our proofs above. Consequently, all the conclusions in \S 4 hold if we replace risk ratios by hazard ratios as in (\ref{eq::hazard-ratio}), and also replace risk differences by hazard differences.

\subsection{Positive Outcome}
For a general positive outcome, we need to modify the definition of confounder-outcome association given the exposure level $A=1$ and the mediator as
\begin{eqnarray}
\label{eq::mean-ratio}
\MR_{UY|(A=1,M, c)} = 
\sup_m
{   \sup_{ u} E(  Y= 1 \mid A=1, m,  c, u )   \over \inf_{ u}  E(  Y= 1 \mid A=1, m,  c, u )  }
\end{eqnarray}
and correspondingly the bounding factor as
$$
\BF_{U|(M,c)} = 
\frac{ \RR_{AU|(M,c)}  \times \MR_{UY|(A=1,M,c)}    }
{  \RR_{AU|(M,c)}  +  \MR_{UY|(A=1,M,c)}   - 1}.
$$

All the conclusions in \S 4 hold if we replace risk ratios by mean ratios as in (\ref{eq::mean-ratio}), and also replace causal risk differences by average causal effects.

\subsection{Without Conditional Independence of $A$ and $U$ Given $C$}

If we drop the assumption $A\ind U\mid C$, then Theorems 1 and 3 still hold if we replace the conditional natural direct effects $\NDE_{\RR|c}^\true$ and $\NDE_{\RD|c}^\true$ by the ones for the unexposed population:
\begin{eqnarray*}
\NDE_{\RR|(A=0, c)}^\true 
&=&
\frac{  \P(  Y_{1M_0} = 1 \mid A=0,  c ) } {  \P(Y_{0M_0} = 1 \mid A=0, c) },\\
\NDE^\true_{\RD|(A=0,c)}
&=&  
\P(  Y_{1M_0} = 1 \mid A=0,  c ) -  \P(Y_{0M_0} = 1 \mid A=0, c)  .
\end{eqnarray*}
We formally state this as a theorem.

\begin{theorem}
\label{thm::unexposed-population}
Under the assumptions in (6), we have the sharp bounds 
$$
\NDE^\true_{\RR| (A=0, c)}  \geq   \NDE^\obs_{\RR|c}/  \BF_{U|(M,c)},
$$
and
$$
\NDE^\true_{\RD| (A=0,c)}  
\geq  \sum_m \PP(Y=1\mid A=1,m,c)   \PP(m\mid A=0,c)  / \BF_{U|(M,c)} 
-  \PP(Y  = 1 \mid A=0, c).
$$
\end{theorem}

\begin{proof}
[of Theorem \ref{thm::unexposed-population}]
The conditional natural direct effect for the unexposed population is
\begin{eqnarray}
&&\NDE^\true_{\RR|(A=0,c)}  \nonumber  \\
%&=& \frac{  \P(  Y_{1M_0} = 1 \mid A=0,  c ) } {  \P(Y_{0M_0} = 1 \mid A=0, c) } 
% \label{eq::nde-1-a0} \\
%\NDE^\true_{\RR|(A=0,c)} 
&=&  \frac{ \int \int   \P(  Y_{1M_0} = 1 \mid A=0,  M_0 = m,  c, u )  F_{(M_0, U)} ( dm, du\mid A=0, c  ) } 
{  \int \int  \P(Y_{0M_0} = 1 \mid A=0, M_0 =m,  c, u)  F_{(M_0, U)} ( dm, du\mid A=0, c  )   }  
\label{eq::nde-2-a0}\\
&=& \frac{ \int \int   \P(  Y_{1m} = 1 \mid  A=0, m,  c, u )  F_{(M_0, U)} ( dm, du\mid A=0,c  ) } 
{  \int \int  \P(Y_{0m} = 1 \mid A=0,m,  c, u)  F_{(M_0, U)} ( dm, du\mid A=0,c  )   }  
\label{eq::nde-3-a0}\\
&=& \frac{ \int \int   \P(  Y_{1m} = 1 \mid    c, u )  F_{M_0}(dm\mid A=0, c,u)   F ( du\mid A=0,c  ) } 
{  \int \int  \P(Y_{0m} = 1 \mid    c, u)   F_{M_0}(dm\mid A=0, c,u)   F ( du\mid A=0, c  )  }  
\label{eq::nde-4-a0}\\
&=& \frac{ \int \int   \P(  Y= 1 \mid A=1, m,  c, u )  F(dm\mid A=0, c,u)   F ( du\mid A=0 , c  ) } 
{  \int \int  \P(Y = 1 \mid A=0,m,   c, u)   F(dm\mid A=0, c,u)   F ( du\mid A=0, c  )   }  
\label{eq::nde-6-a0}\\
&=& \frac{ \int \int   \P(  Y= 1 \mid A=1, m,  c, u )  F(dm, du\mid A=0, c) } 
{  \int \int  \P(Y = 1 \mid A=0,m,   c, u)   F(dm, du\mid A=0, c)    }  
\label{eq::nde-7-a0} ,
\end{eqnarray}
where (\ref{eq::nde-2-a0}) follows from the definition and the law of total probability, (\ref{eq::nde-3-a0}) follows from consistency, (\ref{eq::nde-4-a0}) follows from $Y_{am} \ind A\mid (C,U)$ and $Y_{am} \ind M\mid (A,C,U)$, and the definition of the joint distribution of $(M_0, U)$, (\ref{eq::nde-6-a0}) follows from $Y_{am} \ind A\mid (C,U)$, $Y_{am} \ind M\mid (A,C,U)$ and $M_0\ind A\mid (C,U)$, and (\ref{eq::nde-7-a0}) follows from the definition of the joint distribution of $(M,U)$.

Formula (\ref{eq::nde-7-a0}) is the same as (\ref{eq::nde-8}), and therefore the result in Theorem 1 for $\NDE^\true_{\RR|c} $ also holds for $\NDE^\true_{\RR|(A=0,c)} .$
The discussion for the risk difference scale is analogous.
\end{proof}

\subsection{Decision-Theoretic Framework for Mediation}

The decision-theoretic framework \citep{Didelez::2006, Geneletti::2007} formalizes causal inference by a directed acyclic graph with regime indicators, e.g., Figure \ref{fg::DAG-decision} represents a graph for mediation analysis, where $W$ can be viewed as either $C$ or $(C,U)$.

\begin{figure}[t]
\centering
$$
\begin{xy}
\xymatrix{
& &  W \ar[dll] \ar[drr] \ar[d] \\
A \ar[rr] \ar@/_1.5pc/[rrrr] & & M\ar[rr] & & Y\\
*+[F]{ \sigma_A}  \ar[u] &&  *+[F]{\sigma_M}  \ar[u]
}
\end{xy}
$$
\caption{Directed acyclic graph with regime indicators.}\label{fg::DAG-decision}
\end{figure}
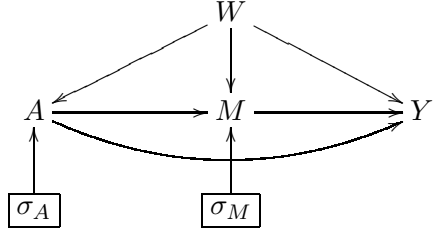

We assume that the conditional independencies among the variables of interest can be represented in a directed acyclic graph, i.e., the joint distribution satisfies the Markov properties. For mediation analysis, we assume that $\sigma_A$, the regime indicator for an intervention on $A$, takes values in $\{\phi, 0, 1  \}$. The no intervention regime $\sigma_A =\phi$ corresponds to observational data, and the atomic intervention $\sigma_A =a$ fixes $A$ at level $a$, for $a=0,1.$ Mathematically, we have
\begin{eqnarray*}
\P( A=a\mid   \pa_A; \sigma_A  = \phi ) = \P(A=a\mid \pa_A),\quad
\P( A=a\mid   \pa_A; \sigma_A  = a^* ) = I(a=a^*).
\end{eqnarray*}
In addition to the values $\phi$ and $m$ for all possible values of $M$, the regime indicator $\sigma_M$ for the mediator also takes the value $G_{a|c}$, a random intervention meaning 
$$
F(dm  \mid \pa_M; \sigma_M = G_{a|c}) = F(dm \mid c; \sigma_A = a, \sigma_M = \phi)\quad 
(a=0,1).
$$
We discuss only the conditional natural direct effect on the risk ratio scale, because the results for conditional natural indirect effect and the risk difference scale are analogous. We define
$$
\NDE_{\RR|c}^\true
=
{ \P(Y=1\mid c; \sigma_A=1, \sigma_M = G_{0|c}) \over \P(Y=1\mid c; \sigma_A=0, \sigma_M = G_{0|c})} . 
$$

We assume $W=C$ and the following sufficient conditions for identification:
\begin{eqnarray*}
C\ind (\sigma_A, \sigma_M),\quad 
Y\ind \sigma_A\mid (A,C; \sigma_M), \quad 
Y\ind \sigma_M\mid (A, M,C; \sigma_A),\quad 
M \ind \sigma_A\mid (A, C).
\end{eqnarray*}
%For interpretations and more general conditions, see \citet{Didelez::2006} and \citet{Geneletti::2007}. 
Under the above conditions, we have
\begin{eqnarray}
&&\P(Y=1\mid c; \sigma_A=a, \sigma_M = G_{a^*|c}) \nonumber \\
%\P(Y=1\mid c; \sigma_A=a, \sigma_M = G_{a^*|c})
&=& \P(Y=1\mid a ,c; \sigma_A=a, \sigma_M = G_{a^*|c})  \label{eq::decision-1} \\
&=& \P(Y=1\mid a ,c;  \sigma_M = G_{a^*|c}) \label{eq::decision-2} \\
&=& \int  \P(Y=1\mid a,m ,c;  \sigma_M = G_{a^*|c})  F(dm\mid a,c;  \sigma_M = G_{a^*|c}) \label{eq::decision-3} \\ 
&=& \int  \P(Y=1\mid a,m ,c )  F(dm\mid c; \sigma_A = a^*, \sigma_M = \phi) \label{eq::decision-4} \\ 
&=& \int  \P(Y=1\mid a,m ,c )  F(dm\mid a^*, c ), \label{eq::decision-5}
\end{eqnarray}
where
(\ref{eq::decision-1}) follows from the definition of $\sigma_A$,
(\ref{eq::decision-2}) follows from $Y\ind \sigma_A\mid (A,C; \sigma_M)$,
(\ref{eq::decision-3}) follows from the law of total probability,
(\ref{eq::decision-4}) follows from $Y\ind \sigma_M\mid (A, M,C; \sigma_A)$ and the definition of $\sigma_M = G_{a^*|c}$,
and (\ref{eq::decision-5}) follows from $M \ind \sigma_A\mid (A, C)$. Therefore, (\ref{eq::decision-5}) implies that the observed version of $\NDE_{\RR|c}^\true$ is
\begin{eqnarray}
\NDE_{\RR|c}^\obs
=
{  \int  \P(Y=1\mid A=1,m ,c )  F(dm\mid A=0, c )
\over 
\int  \P(Y=1\mid A=0,m ,c )  F(dm\mid A=0, c ) }.
\label{eq::nde-obs-decision}
\end{eqnarray}

We assume $W=(C,U)$ and the following conditions for sensitivity analysis:
% with mediator-outcome confounding:
\begin{eqnarray*}
&A\ind  U\mid C,\quad 
(C,U)\ind (\sigma_A, \sigma_M),\\ 
&Y\ind \sigma_A\mid (A,C,U; \sigma_M), \quad 
Y\ind \sigma_M\mid (A, C,U,W; \sigma_A),\quad 
M \ind \sigma_A\mid (A, C,U).
\end{eqnarray*}
Using similar reasoning as in (\ref{eq::decision-1})--(\ref{eq::decision-5}), we have
\begin{eqnarray}
&&\P(Y=1\mid c; \sigma_A=a, \sigma_M = G_{a^*|c}) \nonumber \\
&=&\int  \P(Y=1\mid a ,c, u; \sigma_A=a, \sigma_M = G_{a^*|c}) 
F(du\mid a,c; \sigma_A=a, \sigma_M = G_{a^*|c})  \label{eq::decision-1-u} \\
&=&\int  \P(Y=1\mid a ,c, u; \sigma_M = G_{a^*|c}) 
F(du\mid c; \sigma_A=a, \sigma_M = G_{a^*|c})  \label{eq::decision-2-u} \\
&=&\int \int  \P(Y=1\mid a, m ,c, u; \sigma_M = G_{a^*|c})  
F(dm\mid a, c,u; \sigma_M = G_{a^*|c})
F(du\mid c )  \label{eq::decision-3-u} \\
&=&\int \int  \P(Y=1\mid a, m ,c, u )  
F(dm\mid c,u; \sigma_A = a^*,  \sigma_M = \phi)
F(du\mid c )  \label{eq::decision-4-u} \\
&=&\int \int  \P(Y=1\mid a, m ,c, u )  
F(dm\mid a^*, c,u)
F(du\mid c )  \label{eq::decision-5-u},
\end{eqnarray}
where
(\ref{eq::decision-1-u}) follows from the law of total probability,
(\ref{eq::decision-2-u}) follows from $Y\ind \sigma_A\mid (A,C,U; \sigma_M)$ and the definition of $\sigma_A$,
(\ref{eq::decision-3-u}) follows from the law of total probability and $(C,U)\ind (\sigma_A, \sigma_M)$,
(\ref{eq::decision-4-u}) follows from $Y\ind \sigma_M\mid (A, C,U,W; \sigma_A)$ and the definition of $\sigma_M= G_{a^*|c}$,
and (\ref{eq::decision-5-u}) follows from $M \ind \sigma_A\mid (A, C,U).$
Therefore,
%we can express $\NDE_{\RR|c}^\true$ as
\begin{eqnarray}
\NDE_{\RR|c}^\true
=
{\int \int  \P(Y=1\mid A=1, m ,c, u )  
F(dm\mid A=0, c,u)
F(du\mid c )
\over 
\int \int  \P(Y=1\mid A=0, m ,c, u )  
F(dm\mid A=0, c,u)
F(du\mid c )
}.
\label{eq::true-nde-decision}
\end{eqnarray}

Because the empirical formulas (\ref{eq::nde-obs-decision}) and (\ref{eq::true-nde-decision}) for $\NDE_{\RR|c}^\true$ based on $C$ or $(C,U)$ are both the same as the ones under the potential outcomes framework in the main text, all our results also apply to mediation analysis under the decision-theoretic framework.

\end{document}